\documentclass[12pt]{amsart}
\usepackage[english]{babel}
\usepackage{amsfonts,amssymb,latexsym,amscd}

\theoremstyle{plain}
\newtheorem{theorem}{Theorem}

\newtheorem{proposition}{Proposition}
\newtheorem{corollary}{Corollary}
\theoremstyle{definition}

\theoremstyle{remark}

\usepackage{hyperref}

\usepackage{comment}

\begin{document}

\title{On Orbit Spaces of Distributive Binary $G$-Spaces}

\author{Pavel S. Gevorgyan}
\address{Moscow Pedagogical State University}
\email{pgev@yandex.ru}

\begin{abstract}
The orbit space of a distributive binary $G$-space is studied. A number of its properties in the case of a compact binarily acting group $G$ are established.
\end{abstract}



\maketitle

\section{Introduction}

The notion of a binary $G$-space, or the group of binary transformations of a topological space, was introduced in \cite{Gev}. All binary $G$-spaces and biequivariant maps form a category, which is a natural extension of the category of $G$-spaces and equivariant maps.

The task of carrying over the basic notions of the theory of $G$-spaces to the theory of binary $G$-spaces and studying them is not always easy. Substantial differences arise already in the consideration of bi-invariant sets. For example, a union of bi-invariant subsets of a binary $G$-space is not necessarily a bi-invariant subset.

Difficulties are also involved in the study of orbits of points. The orbit space cannot be constructed for some binary $G$-spaces, because orbits of points may intersect in the general case. However, in the case of a distributive binary $G$-space $X$, the orbits of points have the form $G(x,x)$, $x \in X$ (Proposition \ref{th-distr}), and divide the given space into disjoint classes of subsets (Proposition \ref{prop-5}). This makes it possible to construct the orbit space of a distributive binary $G$-space in the standard way.

We prove certain important properties of the orbit space $X|G$ of a distributive binary $G$-space $X$. In the case of a compact group $G$, the orbit projection $\pi:X\to X|G$ is a closed and proper map (Theorem \ref{th-2}). The orbit space $X|G$ is Hausdorff (Theorem \ref{th-1}), and it is compact (locally compact) if and only if so is $X$ (Theorem \ref{th-3}).

\section{Auxiliary Definitions And Results}

Let $G$ be any topological group, and let $X$ be any topological space. In this paper, all spaces are
assumed to be Hausdorff.

A continuous map $\alpha :G\times X^2\to X$ is called a binary action of the topological group $G$ on the
space $X$ if, for any $g,h\in G$ and any $x,x' \in X$, we have
\begin{equation}\label{eq(1)}
gh(x,x')=g(x, h(x,x')),
\end{equation}
\begin{equation}\label{eq(2)}
e(x,x')=x',
\end{equation}
where $g(x,x') = \alpha (g, x,x')$.

A space $X$ with a fixed binary action of a group $G$, that is, a triple $(G,X,\alpha)$, is called a binary
$G$-space, or a \emph{group of binary transformations}.

Consider the set of all continuous binary operations on a space $X$. On this set, we define an operation “$\ast$” by
\begin{equation}\label{eq4}
(f*\varphi) (x,x')=f(x, \varphi(x,x')),  \quad x,x'\in X,
\end{equation}
where $f, \varphi : X^2 \to X$ is any continuous binary operation on $X$. Thereby, we obtain a topological monoid with the identity element $e(x,x')=x'$. A binary operation   $f : X^2 \to X$ is said to be \emph{invertible} if there exists a binary operation   $\varphi : X^2 \to X$ for which $f\ast \varphi = \varphi \ast f = e$. In this case, $\varphi $ is called the \emph{inverse binary operation} for $f$ and denoted by $f^{-1}$. The following invertibility criterion for binary operations holds.

\begin{theorem}\label{th011}\cite{Gev2}
Let $X$ be a locally compact locally connected space, and let $f:X^2\to X$ be a continuous binary operation. If the map $f_t:X\to X$ defined by $f_t(x)=f(t,x)$ is a homeomorphism for each $t\in X$, then the binary operation $f$ is invertible, and the inverse binary operation $f^{-1}:X^2 \to X$ is defined by $f^{-1}(t,x)=f_t^{-1}(x)$.
\end{theorem}

We denote the group of all invertible continuous binary operations on a topological space $X$ by $H_2(X)$.

A binary action $\alpha$ of a group $G$ on $X$ assigns a continuous binary operation $\alpha_g:X^2 \to X$ on $X$ to each element $g \in G$ by the formula $\alpha_g(x,x') = \alpha (g,x,x')$. The correspondence $g\to \alpha_g$ is a continuous homomorphism from the group $G$ to the group $H_2(X)$ of all continuous invertible binary operations on $X$ \cite{Gev2}. Thus, the elements of $G$ can be treated as continuous binary operations on $X$ acting by the rule $g(x,x')=\alpha_g(x,x')$.

Let  $t\in X$ be any element. Consider the continuous map $\alpha_t:G\times X\to X$ defined by
\begin{equation}\label{eq1}
\alpha_t(g,x)= \alpha(g,t,x).
\end{equation}
The map $\alpha_t$ is an action of the group $G$ on the space $X$. Thus, a binary action $\alpha$ of a group $G$ on a space $X$ induces a family of “ordinary” actions $\{\alpha_t\}$, $t\in X$, of the group $G$ on $X$.

Let $(G,X,\alpha)$ and $(G,Y,\beta)$ be binary $G$-spaces. A continuous map $f:X\to Y$ is said to be \emph{biequivariant}, or \emph{$G$-biequivariant}, if
$$f(\alpha(g,x,x'))=\beta(g,f(x),f(x')),$$
%,
or %
$$f(g(x,x'))=g(f(x),f(x')),$$
for all $g\in G$ and $x,x' \in X$.

A biequivariant map $f:X\to Y$ which is simultaneously a homeomorphism is called an \emph{equivalence} of binary $G$-spaces, or a \emph{biequimorphism}.

All binary $G$-spaces and biequivariant maps form a category, which we denote by $G-Top^2$.

It is easy to show that any biequivariant map $f:X\to Y$ is equivariant with respect to the induced
actions $\alpha_t$ and $\beta_{f(t)}$ for each $t \in X$.

Now let $X$ be any $G$-space. We define a binary $G$-action on $X$ by
\begin{equation}\label{eq11}
g(x,x')=g x',
\end{equation}
for all  $g\in G$ and $x, x'\in X$.

Note that if $X$ and $Y$ are $G$-spaces, then any equivariant map $f:X\to Y$ is biequivariant with respect
to the action \eqref{eq11}. Therefore, the category  $G-Top^2$ is a natural extension of the category $G-Top$ of all $G$-spaces and equivariant maps.

Let $A$ and $B$ be subsets of a binary $G$-space $X$, and let $K$ be a subset of the group $G$. We put 
$$K(A,B) = \{g(a,b); \ g\in K, a\in A, b\in B\}.$$

A subset $A$ of a binary $G$-space $X$ is said to be \emph{bi-invariant}, or \emph{$G$-bi-invariant}, if $G(A, A) = A$. A bi-invariant subset $A\subset X$ is called a \emph{binary G-subspace} of $X$.

If $G$ is a compact group and $U\subset X$ is an open subset of $X$, then $G(U,U)$ is an open subset of $X$ as well. The following theorem is valid.

\begin{theorem} \label{th-zamkn}
Let $G$ be a compact group, and let $X$ be a binary $G$-space. If $A$ is a closed subset of $X$, then $G(A, A)$ is closed in $X$.
\end{theorem} 

\begin{proof}
Let $x_0\in X$ be a point in the closure of $G(A,A)$. Since the set $A\subset X$ is closed and 
$$G(A,A)= \bigcup\limits_{a\in A}G(a,A)$$ 
it follows that, for some point $a_0\in A$, $x_0$ belongs to the closure of $G(a_0,A)$.

Note that, to prove the closedness of $G(A, A)$, is suffices to prove the existence of an element $g_0\in G$
for which $g_0 (a_0, x_0)$ belongs to the closure of $A$ and, therefore, $g_0 (a_0, x_0) \in A$, because $A$ is closed. Indeed, for this element, we have
$$x_0=g_0^{-1}(a_0, g_0 (a_0,x_0))\in G(A,A).$$

Suppose that, on the contrary, $g (a_0, x_0)$ belongs to the closure of $A$ for no $g\in G$, i.e., any point $g\in G$ has a neighborhood $W_{g(a_0,x_0)}$ such that
$$W_{g(a_0,x_0)}\cap A = \emptyset.$$
Since the binary action of the group $G$ on $X$ is continuous, there exist neighborhoods $g\in V_g$ and $x_0\in U_{x_0}(g)$, for which $V_g(a_0,U_{x_0}(g))\subset W_{g(a_0,x_0)}$ and, therefore,
$$V_g(a_0,U_{x_0}(g))\cap A = \emptyset.$$
The family $\{V_g, \ g\in G\}$ is an open cover of the compact group $G$, and it has a finite subcover $V_{g_1}, V_{g_2}, \ldots , V_{g_n}$. We set
$$U_{x_0} = U_{x_0}(g_1)\cap \ldots \cap U_{x_0}(g_n).$$
Note that $G(a_0,U_{x_0})\cap A = \emptyset$, hence $U_{x_0}\cap G(a_0,A) = \emptyset$, i.e., $x_0$ does not belong to the closure of  $G(a_0,A)$, which contradicts the assumption.
\end{proof}

A binary $G$-space $X$ is said to be \emph{distributive} if, for any $x, x', x'' \in X$ and $g, h\in G$, we have %
\begin{equation}\label{eq1-1}
g(h(x,x'), h(x,x''))=h(x,g(x', x'')).
\end{equation}

Let $H$ be a subgroup of the group $G$. The continuous map $\alpha : H\times G^2\to G$ defined by
\begin{equation}\label{eq-H on G}
\alpha(h,x,y) = xhx^{-1}y \quad \text{or} \quad h(x,y) = xhx^{-1}y,
\end{equation}
where $h\in H$ and $x,y\in G$ are any elements, is a distributive binary action of the group $H$ on $G$.

The class of distributive binary $G$-spaces plays an important role in the theory of binary $G$-spaces. 

A detailed exposition of these and other definitions, notions, and results used in the paper without
reference can be found in \cite{Bredon}, \cite{Engelking} and \cite{Gev}.

\section{Orbits And The Orbit Space Of A Distributive Binary G-Space}

Let $X$ be a binary $G$-space. The \emph{orbit} of a point $x\in X$ is defined as the minimal bi-invariant subset $[x]$ of $X$ containing the point $x$. Obviously, $x\in G(x,x)\subset [x]$ for all $x \in X$. Therefore, if $G(x, x)$ is a bi-invariant set, then $G(x, x) = [x]$.

As is known \cite[Example 2]{Gev-Naz}, sets of the form $G(x, x)$ are not generally bi-invariant. However, the following assertion is valid \cite{Gev2}.

\begin{proposition}\label{th-distr}
Let $X$ be a distributive binary $G$-space. Then, for any  $x\in X$, the set $G(x,x)$ is bi-invariant.
\end{proposition}

\begin{proof}
We must prove that $g(h(x,x), k(x,x))\in G(x,x)$ for any $h(x,x), k(x,x)\in G(x,x)$ and any $g\in G$. Indeed, by virtue of \eqref{eq1-1}, we have
\begin{multline*}
g(h(x,x), k(x,x))=g(h(x,x), h(x,h^{-1}k(x,x)))=h(x, g(x, h^{-1}k(x,x)))= \\
= h(x, gh^{-1}k(x, x))=hgh^{-1}k(x, x)\in G(x,x). 
\end{multline*}
\end{proof}

Thus, in a distributive binary $G$-space $X$, the set $G(x, x)$ is the orbit of $x\in X$.

An example constructed in \cite{Gev-Naz} shows that orbits in a binary $G$-space may intersect, i.e., the orbit of a point $x$ in a binary $G$-space $X$ is not necessarily the orbit of each of its points. It may contain smaller orbits. Therefore, we cannot define the notion of orbit space for an arbitrary binary $G$-space. However, as shown by Proposition below, the situation in distributive binary $G$-spaces is different.

\begin{proposition}\label{prop-5}
Any two orbits in a distributive binary $G$-space $X$ either are disjoint or coincide.
\end{proposition}

\begin{proof}
Let $x$ be any element of a distributive binary $G$-space $X$. By Proposition \ref{th-distr}, the set $G(x,x)$ is
bi-invariant; therefore, this set is the orbit of the point $x\in X$: $[x]=G(x,x)$.

Let us prove that the set $[x]$ is the orbit of each of its points, i.e., does not contain bi-invariant proper subsets. Take any element $g_0(x,x)\in [x]$. Let us show that $[g_0(x,x)]=[x]$. Obviously, we have $[g_0(x,x)]\subset [x]$. Therefore, it suffices to prove that $[x] \subset [g_0(x,x)]$, i.e., any element $g(x,x)$ of the orbit $[x]$ is also an element of the orbit $[g_0(x,x)]$: 
$$g(x,x) = g'(g_0(x,x), g_0(x,x))$$ 
for some $g'\in G$. Indeed, in view of the distributivity of the binary action, we have
\begin{multline*} 
g(x,x)= g_0g_0^{-1}g(x,x) = g_0(x, g_0^{-1}g(x,x))  =  \\
=g_0^{-1}g(g_0(x,x), g_0(x,x)) = g'(g_0(x,x), g_0(x,x)),
\end{multline*}
where $g'=g_0^{-1}g$.

Now let $[x]$ and $[x']$ be any orbits of a distributive binary $G$-space $X$. Suppose that these orbits intersect, i.e., there exists a point $\widetilde{x}\in [x]\cap [x']$. Then, as proved above, we have $[x]=[\widetilde{x}]$ and $[x']=[\widetilde{x}]$. Therefore, the orbits $[x]$ and $[x']$ coincide.
\end{proof} 

Thus, any distributive binary $G$-space $X$ is partitioned into disjoint classes of subsets, namely, into the orbits of points. We denote the corresponding quotient set by $X|G$. Let $\pi=\pi_X:X\to X|G$ be the natural quotient map, which takes each point $x\in X$ to its orbit $[x]$. We set $x^*=\pi([x])$. Endowing the set $X|G$ with the quotient topology, we obtain a space called the \emph{orbit space} of the distributive binary $G$-space $X$.

If $X$ and $Y$ are distributive binary $G$-spaces, then any biequivariant map $f:X\to Y$ induces the map $\pi(f)=f^*:X|G \to Y|G$ of the orbit spaces defined by $f^*(x^*)=(f(x))^*$. This map is well defined, because is does not depend on the choice of an orbit representative:
$f^*((g(x,x))^*) = (f(g(x,x)))^*=(g(f(x),f(x)))^*=(f(x))^*=f^*(x^*)$ for any $g\in G$.

It is easy to see that $\pi(1_X)=1_{X|G}$ and $\pi(gf)=g^*f^*$, where $f:X\to Y$ and $g:Y\to Z$ are biequivariant maps of distributive binary $G$-spaces. Therefore, the formulas $\pi(X)=X|G$ and $\pi(f)=f^*$ define a covariant functor from the category of distributive binary $G$-spaces and biequivariant maps to the category of topological spaces and continuous maps.

\begin{theorem}\label{prop-1}
Let $X$ be a distributive binary $G$-space, and let $K$ be a compact subset of the group $G$. Then, for any closed subset $A$ of $X$, the set $ KA = \{g(a,a); g\in K, a\in A\}$ is closed in $X$.
\end{theorem}

The proof of this theorem is based on the following proposition.

\begin{proposition}\label{prop-4}
Let $X$ be a distributive binary $G$-space. Then, for any element $g\in G$, the map $\Delta_g:X\to X$ defined by
\begin{equation}\label{eq-1}
\Delta_g(x) = g(x,x), \quad x\in X,
\end{equation}
is a homeomorphism.
\end{proposition}

\begin{proof}
The continuity of $\Delta_g$ is obvious.

Let $\Delta_g(x)=\Delta_g(x')$, $x,x'\in X$. Then
\[
g^{-1}(\Delta_g(x),\Delta_g(x)) = g^{-1}(\Delta_g(x'),\Delta_g(x')).
\]
For the left-hand side of this relation, the definition \eqref{eq-1} and the distributivity of the binary action of $G$ on $X$ imply
\begin{equation*}
g^{-1}(\Delta_g(x),\Delta_g(x)) = g^{-1}(g(x,x),g(x,x)) = g(x,g^{-1}(x,x))=e(x,x) = x.
\end{equation*}
Similarly,
\begin{equation*}
g^{-1}(\Delta_g(x'),\Delta_g(x')) = g^{-1}(g(x',x'),g(x',x')) = g(x',g^{-1}(x',x'))= e(x',x') = x'.
\end{equation*}
Thus, $x=x'$, and hence the map $\Delta_g$ is injective.

Now note that the map $\Delta_{g^{-1}}:X\to X$ is inverse to $\Delta_g$. Indeed, in view of distributivity, we have %
\[
\Delta_g\Delta_{g^{-1}}(x)=  \Delta_g(g^{-1}(x,x)) = g(g^{-1}(x,x),g^{-1}(x,x)) = g^{-1}(x, g(x,x)) = e(x,x) = x
\]
for any $x\in X$. Similarly, $\Delta_{g^{-1}}\Delta_g(x)=x$.
\end{proof}

\begin{proof}[Proof of Theorem \ref{prop-1}]
%Proof of Theorem 3. 
Take $x_0\in X \backslash KA$ and consider any element $g\in K$. By Proposition \ref{prop-4} the set $\Delta_g(A)=gA=\{g(a,a); a\in A\}$ is closed in $X$. Since $gA\subset KA$, it follows that $x_0$ belongs to the open set $X\backslash gA$. Therefore, the relation $e(x_0,x_0)=x_0$ and the continuity of the binary action imply the existence of neighborhoods $V_g$ of $e$ and $Ug$ of $x_0$ such that
$$V_gU_g=\{h(x,x); h\in V_g, x\in U_g\}\subset X\backslash gA,$$ 
i.e., 
$$V_gU_g\cap gA = \emptyset.$$ 
Therefore,
\[
U_g\cap V_g^{-1}gA  = \emptyset.
\]

The family $\{V_g^{-1}g; g\in K\}$ is an open cover of the compact set $K$. Let $V_{g_1}^{-1}g_1$, $V_{g_2}^{-1}g_2$, \ldots, $V_{g_n}^{-1}g_n$ be its finite subcover: 
\[
K\subset V_{g_1}^{-1}g_1\cup V_{g_2}^{-1}g_2\cup \ldots \cup V_{g_n}^{-1}g_n.
\]

We set
\[
U=\bigcap_{i=1}^n U_{g_i}.
\]
This is an open neighborhood of $x_0$, and
\begin{equation}\label{eq-5}
U\cap V_{g_i}^{-1}g_iA  = \emptyset
\end{equation}
for all $i=1, \ldots, n$.

Let us prove that either $U\subset X\backslash KA$ or $U\cap KA = \emptyset$. Indeed, for any element $g(a,a)\in KA$, there exists a number $i$, $1\leqslant i\leqslant n$, such that $g\in V_{g_i}^{-1}g_i$, and hence $g(a,a)\notin U$ by virtue of \eqref{eq-5}.
\end{proof}

Theorem \ref{prop-1} directly implies the following result.

\begin{corollary}\label{cor-1} 
Let $G$ be a compact topological group, and let $X$ be a distributive binary $G$-space. Then, for any closed subset $A$ of $X$, the union 
$$GA = \bigcup\limits_{a\in A}G(a,a)$$ 
of the orbits of points in $A$ is closed in $X$.
\end{corollary}

\begin{theorem}\label{th-2}
Let $G$ be a compact topological group, and let $X$ be a distributive binary $G$-space. Then the orbit projection $\pi:X\to X|G$ is

(a) a closed map;

(b) a proper map.
\end{theorem}

\begin{proof}
To prove (a), consider any closed subset $A$ of $X$. By the definition of the topology of $X|G$, the closedness of $\pi(A)\subset X|G$ is equivalent to that of $\pi^{-1}\pi(A)\subset X$. The set
$$\pi^{-1}\pi(A)= \bigcup\limits_{a\in A}G(a,a)$$ 
is the union of orbits of points in the closed set $A$, and hence it is closed by Corollary \ref{cor-1}.

As is known, a closed map is proper if the preimages of all points are compact. It follows that the map $\pi:X\to X|G$ is proper, because it is closed and the preimages of all points of the orbit space $X|G$ are compact.
\end{proof}

\begin{theorem}\label{th-1}
Let $G$ be a compact topological group, and let $X$ be a distributive binary $G$-space. Then the orbit space  $X|G$ is Hausdorff.
\end{theorem}

\begin{proof}
Let $x^*$ and $x'^*$ be any two different orbits in the space $X|G$: $x^*\neq x'^*$. The sets 
$$\pi^{-1}(x^*)=G(x,x), \quad \pi^{-1}(x'^*)=G(x',x')$$
are disjoint compact subsets of the Hausdorff space $X$; therefore, they have disjoint open neighborhoods. In particular, there exists a neighborhood $U$ of $x$ for which
$$\overline{U}\cap G(x',x') = \emptyset .$$
Hence $x'^*\notin \pi (\overline{U})$, and $\pi(\overline{U})$ is a closed subset of $X|G$ by Theorem \ref{th-2}. Thus, the open sets $\pi(U)$ and
$X|G \backslash \pi(\overline{U})$ are disjoint neighborhoods of $x^*$ and $x'^*$, respectively.
\end{proof}

\begin{theorem}\label{th-3}
Let $G$ be a compact topological group, and let $X$ be a distributive binary $G$-space. Then the orbit space $X|G$ is compact (locally compact) if and only if so is the space $X$.
\end{theorem}

\begin{proof}
Suppose that $X$ is compact. Then the orbit space $X|G=\pi(X)$ is compact, since this is a continuous image of $X$. Conversely, let $X|G$ be compact. By Theorem \ref{th-2} the orbit projection $\pi:X\to X|G$ is a proper map. Therefore, $X$ is compact.

Let $X$ be a locally compact space. Take any element $x^*\in X|G$. Since $X$ is locally compact, the point $x\in X$ has an open neighborhood $U$ with compact closure $\overline{U}$. Its image $\pi(U)$ is an open neighborhood of $x^*$ with compact closure $\pi(\overline{U})$. Conversely, suppose that $X|G$ is locally compact. Take any element $x\in X$. Let $U\subset X|G$ be an open neighborhood of $x^*=\pi(x)\in X|G$ with compact closure $\overline{U}$. Then $\pi^{-1}(U)\subset X$ is an open neighborhood of $x$ with compact closure $\pi^{-1}(\overline{U})$ by Theorem \ref{th-2}.
\end{proof}

\end{document}